\documentclass[reqno]{proc-l}
\usepackage{amsfonts,amssymb,amsmath,amsthm,amsopn,amsxtra,amstext,amsbsy}
\usepackage[cp1251]{inputenc}
\allowdisplaybreaks[4]
\usepackage[active]{srcltx}
\renewcommand{\subjclassname}{\textup{2010} Mathematics Subject Classification}
\numberwithin{equation}{section}
\newtheorem{theorem}{Theorem}[]

\newtheorem{lemma}{Lemma}[]

\newtheorem{corollary}{Corollary}[]
\newtheorem{definition}{Definition}[]
\newtheorem{thmx}{Theorem}

\newtheorem{lemx}{Lemma}

\DeclareSymbolFont{AMSb}{U}{msb}{m}{n}
\DeclareMathAlphabet{\Bb}{U}{msb}{m}{n}
\newcommand{\fo}[1]{{\mbox{\footnotesize{$#1$}}}}
\DeclareMathOperator*{\limsp}{\overline{\lim}}
\DeclareMathOperator{\card}{\rm{card}}

\DeclareMathOperator{\sign}{\rm{sign}}
\usepackage{hyperref}

\begin{document}
\title[Smoothing of weights ...]{ Smoothing of weights in the Bernstein approximation problem}
\author[Andrew Bakan]{Andrew Bakan}
\address{Institute of Mathematics, National Academy of
Sciences of Ukraine, Kyiv 01601, Ukraine}%
\email{andrew@bakan.kiev.ua}%
\author[J\"urgen Prestin]{J\"urgen Prestin}%
\address{Institut f\"ur Mathematik, Universit\"at zu L\"ubeck,
D-23562 L\"ubeck, Germany}%
\email{prestin@math.uni-luebeck.de}%

\date{}


\thanks{The first author acknowledges  support from the German Academic Exchange Service (DAAD, Grant 57210233).}

\begin{abstract} In 1924  S.~Bernstein \cite{bern}  asked for conditions on a uniformly boun\-ded on $\Bb{R}$ Borel function (weight)
$w: \Bb{R} \to [0, +\infty )$  which imply the denseness of algebraic polynomials ${\mathcal{P} }$  in the seminormed  space
$ C^{\,0}_{w} $ defined  as the linear set $ \{f \in C (\Bb{R}) \ | \ w (x) f (x) \to 0 \ \mbox{as} \ {|x| \to +\infty}\}$ equipped with the seminorm
$\|f\|_{w} := \sup\nolimits_{x  \in  {\Bb{R}}}   w(x)| f( x )|$. In 1998 A.~Borichev and M.~Sodin \cite{bs} completely solved this problem for
all those weights $w$ for which ${\mathcal{P} }$ is dense in $ C^{\,0}_{w} $  but there exists a positive integer $n=n(w)$ such that
${\mathcal{P} }$ is not dense in $ C^{\,0}_{(1+x^{2})^{n} w}$. In the present paper we establish that if ${\mathcal{P} }$ is dense in
$ C^{\,0}_{(1+x^{2})^{n} w}$ for all $n \geq 0$ then for arbitrary $\varepsilon > 0$ there exists a weight
$W_{\varepsilon} \in C^{\infty} (\Bb{R})$ such that ${\mathcal{P}}$ is dense in  $C^{\,0}_{(1+x^{2})^{n} W_{\varepsilon}}$
for every $n \geq 0$ and  $W_{\varepsilon} (x) \geq w (x) + \mathrm{e}^{- \varepsilon |x|}$ for all $x\in \Bb{R}$.
\end{abstract}
\maketitle

\section{Introduction}

\vspace{0.25cm}
Let $C(\Bb{R})$ be the linear space of all continuous real-valued functions on $\Bb{R}$,
$\mathcal{W} (\Bb{R})$ the set of all uniformly bounded on $\Bb{R}$ Borel functions $w : \Bb{R} \to \Bb{R}^{+}:= [0, +\infty )$
which have an unbounded support $S_w := \left\{x \in {\Bb{R}} \, | \, w (x) > 0\right\} $ and
satisfy $|x|^{n}  w (x) \to 0$ as $|x| \to \infty$ for all $n \in \Bb{N}_{0} := \{0, 1, 2, ...   \}$.
Denote by $\mathcal{P} $ the set of all algebraic polynomials with real coefficients and by $C^{\infty}(\Bb{R})$ the family of all real-valued infinitely
continuously differentiable functions on $\Bb{R}$.

For $w \in \mathcal{W}(\Bb{R})$ the seminormed space $C^{\,0}_w (\Bb{R})$ consists of
the linear set of all  $f \in C(\Bb{R})$ with $\lim_{|x| \to +\infty}  w (x) \, f (x) = 0$
 and the semi-norm $\|\cdot\|_{w}$, where $\| f \|_{w} := \sup_{x \, \in \, {\Bb{R}}} w(x)\, |f( x )|$.

We recall the definition of the so-called {\emph{upper Baire function}} $M_{F}$ of $F : \Bb{R} \to \Bb{R}$
as $M_{F } (x) := \lim_{\delta \downarrow 0}
\sup_{y \in (x-\delta,  x+\delta) } F(y)$ (see \cite[p.~129]{nat}). If $F$ is locally bounded from above, then
$M_{F}$ is an upper semi-continuous function and $F (x) \leq M_{F} (x)$, $x \in \Bb{R}$.
It is easy to verify  that for arbitrary $-\infty < A < B < +\infty$,
$w \in \mathcal{W} (\Bb{R})$ and $f \in C (\Bb{R})$ we have
\begin{equation*}
    \sup\limits_{x \in (A, B)} w (x) \left| f (x)\right| =  \sup\limits_{x \in (A, B)} M_{w} (x) \left| f (x)\right|.
\end{equation*}
This means that the seminormed spaces $C^{\,0}_{w}(\Bb{R})$ and $C^{\,0}_{M_{w}}(\Bb{R})$
coincide identically and, in particular, $\mathcal{P}$ is dense in $C^{\,0}_{w}(\Bb{R})$ iff it is dense in
$C^{\,0}_{M_{w}}(\Bb{R})$. Thus, it is possible to assume everywhere below that $w \in \mathcal{W}^{*} (\Bb{R})$
where  $\mathcal{W}^{*} (\Bb{R})$ denotes the family of all those $w \in \mathcal{W}(\Bb{R})$
which are upper semi-continuous on $\Bb{R}$, i.e., $M_{w}(x) \equiv w(x)$ for all $x \in \Bb{R}$.

Introduce
\begin{equation}\label{e7}
   \mathcal{W}^{{ \rm{dens}}} (\Bb{R}) := \left\{
   w \in \mathcal{W}^{*} (\Bb{R}) \ | \ \mathcal{P} \ \mbox{is dense in}\  C^{\,0}_{w} (\Bb{R}) \right\}.
\end{equation}

In 1924 S.~Bernstein \cite{bern} asked for conditions on $w \in \mathcal{W}^{*} (\Bb{R})$
to be in $ \mathcal{W}^{{\rm{dens}}}(\Bb{R})$. This problem is known as {\emph{Bernstein's approximation problem}}.
Various results towards a final solution of Bernstein’'s approximation problem have been obtained independently by
L.~Carleson \cite{car}(1951), H.~Pol\-lard \cite{pol}(1953), S.~N.~Mer\-ge\-lyan \cite{m}(1958) and L.~de~Bran\-ges
\cite{db}(1959) (see also the surveys of P.~Koo\-sis \cite{k},  A.~Pol\-to\-rat\-ski \cite{polt} and M.~So\-din \cite{s}).

The solution of Bernstein's problem given by L.~de~Branges \cite{db} in 1959
was slightly improved in 1996  by M.~Sodin and P.~Yuditskii \cite{s1} and attained the following form.

Let $f$ be an entire function, $\Lambda_f$ be the set of all its zeros, $0 \leq r, \rho < \infty$ and
$\sigma_f (\rho) := {\overline{\lim}}_{r\to\infty} \ r^{-\rho} \, \log M_f (r) $, where $M_f (r) :=
\sup_{|z| = r} | f(z) |$. We say that $f$ is of {\it{minimal exponential type}} if $\sigma_f (1) = 0$.
Denote by $\mathcal{E}_0(\Bb{R})$  the family of all entire functions $f$ of minimal
exponential type which are  real on the real axis (in short: real) and  have only real simple zeros.

\begin{thmx}[L. de Branges, 1959 {\cite{db}}]\label{tha}
Let $w \in \mathcal{W}^{*} (\Bb{R} )  $. Then  $\mathcal{P}$ is not dense in $C^{\,0}_{w}(\Bb{R})$ if and only
if there exists an entire function $B \in \mathcal{E}_0 (\Bb{R}) $  such that
$\Lambda_{B} \subset S_w = \left\{x \in {\Bb{R}} \, | \, w (x) > 0   \right\}$ and
\begin{equation*}
\sum_{{\lambda \in \Lambda_B}  }  \ \frac{1}{w ( \lambda ) |B^{\,\prime } (\lambda )| } \ < \ + \infty \ .
\end{equation*}
\end{thmx}

In 1958 S.~Mergelyan \cite{m}  proved that if algebraic polynomials are dense in $ C^{\,0}_{w}(\Bb{R}) $ but are not
dense in $ C^{\,0}_{(1+x^{2})^{n} w}(\Bb{R}) $ for some positive integer $n$, then $w$ has countable support and
the number of points in the set $\{ x \in \Bb{R} \ | \ w (x) > 0  , \ |x|< R\}$ is $o(R)$ as $R \to +\infty$. Motivated
by this result, A.~Borichev and M.~Sodin in 1998 \cite{bs} divided Bernstein's approximation problem into two parts.

\begin{definition}{\rm{ Let
$ w \in \mathcal{W}^{*} (\Bb{R})$. It is said that algebraic polynomials $\mathcal{P}$ are {\emph{regularly dense}}
in $C^{\,0}_{w} (\Bb{R})$ if they are dense in $C^{\,0}_{(1+x^{2})^{n} w}(\Bb{R})$ for all $n \in \Bb{N}_{0}$.

Algebraic polynomials  $\mathcal{P}$ are called to be {\emph{singularly dense}} in $C^{\,0}_{w} (\Bb{R})$ if they are
dense in $C^{\,0}_{w}(\Bb{R})$ but not in $C^{\,0}_{(1+x^{2})^{n} w} (\Bb{R})$ for a certain $n \in \Bb{N}: = \{ 1, 2, ... \}$.}}
\end{definition}

Similarly to \eqref{e7}, we denote
\begin{eqnarray*}
   \mathcal{W}^{{\rm{reg}}} (\Bb{R}) &:=& \left\{
   w \in \mathcal{W}^{*} (\Bb{R}) \ | \ \mathcal{P} \ \mbox{is regularly dense in}\ C^{\,0}_{w} (\Bb{R}) \right\}, \\
      \mathcal{W}^{{\rm{sing}}} (\Bb{R}) &:=& \left\{
   w \in \mathcal{W}^{*} (\Bb{R}) \ | \ \mathcal{P}  \ \mbox{is singularly dense in}\ C^{\,0}_{w} (\Bb{R}) \right\}.
\end{eqnarray*}

It is obvious that $ \mathcal{W}^{{\rm{reg}}} (\Bb{R})$ and $\mathcal{W}^{{\rm{sing}}} (\Bb{R})$
are two non-intersecting classes of weights and
\begin{equation*}
  \mathcal{W}^{{\rm{dens}}} (\Bb{R}) =  \mathcal{W}^{{\rm{reg}}} (\Bb{R}) \sqcup \mathcal{W}^{{\rm{sing}}} (\Bb{R}) \ ,
\end{equation*}
where the symbol $\sqcup$ denotes the union of two non-intersecting sets.

Thus, the finding of conditions on a given weight $ w \in \mathcal{W}^{*} (\Bb{R})$ to be in $ \mathcal{W}^{{\rm{reg}}} (\Bb{R})$
or in $\mathcal{W}^{{\rm{sing}}} (\Bb{R})$  divides Bernstein's approximation problem into two independent parts:
regular and singular, respectively.
A complete solution of the singular part was given by A.~Borichev and M.~Sodin \cite{bs} in 1998.

\begin{thmx}\label{thb}
Let $w\in\mathcal{W}^{*}(\Bb{R})$. Algebraic polynomials $\mathcal{P}$ are singularly dense in $C^{\,0}_{w} (\Bb{R})$
if and only if $w$ is discrete and there exist an entire function $E \in \mathcal{E}_0(\Bb{R})$ and a nonnegative
integer $n$ such that
\begin{equation*}
  w(x) = \sum\limits_{\lambda \in \Lambda_{E}} w(\lambda) \, \chi_{\lambda }(x), \ \ x \in \Bb{R}, \ \ \ \   \chi_{\lambda}(x)
   :=  {{\left\{ \begin{array}{ll}
       0, &\ \ \textup{if} \ \ x \neq \lambda, \\
       1, &\ \ \textup{if} \ \ x = \lambda,
      \end{array}\right. }}
\end{equation*}
\begin{equation*}
\sum_{\lambda \in \Lambda_{E} } \frac{1}{(1+\lambda^{2})^{k} \ w(\lambda) \left| E^{\,\prime} (\lambda )\right|} \ \
\left\{
  \begin{array}{ll}
  < + \infty, & \ \ \textup{if} \ \  k = n +1, \\
  =   + \infty, & \ \ \textup{if} \ \  k = n,
  \end{array}\right.
\end{equation*}
and
\begin{equation*}
   \sum_{\lambda \in \Lambda_{F} } \frac{1}{w \left(\lambda\right) \left| F^{\,\prime} (\lambda )\right|} = + \infty
\end{equation*}
for arbitrary transcendental entire functions $F$ of minimal exponential type such that
$\Lambda_{F} \subset \Lambda_{E}$ and $E/F$ is transcendental.
\end{thmx}

The regular part of Bernstein's approximation problem is still open but the following important result holds.
\begin{thmx}[M.~Sodin, 1996 \cite{s}]\label{thc}
 If $w\in \mathcal{W}^{{\rm{reg}}} (\Bb{R})$, then
$w (x) + \mathrm{e}^{- \delta |x|}\in \mathcal{W}^{{\rm{reg}}}(\Bb{R}) $ for every $\delta > 0$.
\end{thmx}

The following statement about perturbations of zeros of an entire function was proved in \cite[Lemma 5, p.~237]{b5} (2005).

\begin{lemx}\label{lma}
For an arbitrary entire function $B \in \mathcal{E}_{0} (\Bb{R})$ with zeros $\Lambda_B = \left\{ b_n \right\}_{n \geq 1}$
there exists a constant $C>0$ and a sequence of real positive numbers $\left\{\delta_n \right\}_{n \geq 1}$ such
that for any sequence of real numbers $\left\{d_n \right\}_{n \geq 1}$ satisfying
\begin{equation*}
|b_n - d_n| \leq \delta_n, \quad \ n \geq 1, \
\end{equation*}
one can find an entire function $D \in \mathcal{E}_{0}(\Bb{R})$ such that $\Lambda_D = \left\{d_n\right\}_{n \geq 1}$ and
\begin{equation*}
|B^{\,\prime}(b_n )| \leq C \cdot |D^{\,\prime }(d_n )|, \quad n \geq 1.
\end{equation*}
\end{lemx}

If the set of real numbers $\{|B^{\,\prime} (b_n )|\}_{n\geq 1}$ in Lemma~\ref{lma} is bounded from below,
then the result of Lemma~\ref{lma} can be improved as follows.

\begin{lemma}\label{lm1}
Let $B \in \mathcal{E}_{0} (\Bb{R})$  and  $\Lambda_{B}$ denote the set of its zeros. Assume that
\begin{equation}\label{8j}
  \sum_{\lambda \in \Lambda_{B}} \frac{1}{ \left|B^{\,\prime} (\lambda)\right|} < \infty.
\end{equation}
Then, for arbitrary $\delta > 0$ there exist constants $C_{\delta} = C_{\delta}(B), \rho_{\delta}=\rho_{\delta}(B)  > 0$
such that for any set of real numbers $\left\{d_{\lambda}\right\}_{\lambda \in \Lambda_{B} }$ satisfying
\begin{equation}\label{9j}
   \left|\lambda - d_{\lambda} \right| \leq \rho_{\delta} \, \mathrm{e}^{{{-\delta |\lambda|}}}, \ \ \lambda \in \Lambda_{B},
\end{equation}
one can find an entire function $D \in \mathcal{E}_{0}(\Bb{R})$ such that
$\Lambda_{D}  = \left\{d_{\lambda}\right\}_{\lambda \in \Lambda_{B}}$ and
\begin{equation}\label{11j}
   \left|B^{\,\prime}(\lambda)\right| \leq C_{\delta} \, \left|D^{\,\prime}(d_{\lambda})\right|,\ \ \lambda \in \Lambda_{B}.
\end{equation}
\end{lemma}

\bigskip Observe that the proof of Lemma~\ref{lm1}  in Section~\ref{lemma1} gives the explicit expressions for the constants $\rho_{\delta}$ and $C_{\delta}$
in \eqref{9j} and in \eqref{11j}.
Lemma~\ref{lm1} is instrumental for the proof of the next statement.
\begin{lemma}\label{lm2}
 Let $\varepsilon > 0$ and $w\in \mathcal{W}^{{\rm{reg}}} (\Bb{R})$.
Then,
\begin{equation}\label{th1f1}
  w_{\varepsilon}(x) :=   \sup_{|t| \le \mathrm{e}^{-\varepsilon |x|}}\left( w (x + t) +
  \mathrm{e}^{-\varepsilon |x+t|}\right) \in \mathcal{W}^{{\rm{reg}}} (\Bb{R}).
\end{equation}
\end{lemma}
\begin{proof}
In view of \cite[Example 1, p.~8]{h},  the function
\begin{equation}\label{plem81za}
 \beta_{\varepsilon}(x):= w (x) + \mathrm{e}^{-\varepsilon |x|}
\end{equation}
is upper semi-continuous on $\Bb{R}$ and an application of \cite[Theorem 1.2, p.~4]{h}
to the supremum in \eqref{th1f1} yields for each $x\in \Bb{R}$ the existence of  $\theta_{\varepsilon} (x)\in [-1,1]$ such that
\begin{equation}\label{plem81zb}
    w_{\varepsilon} (x) = \beta_{\varepsilon} \left(x + \theta_{\varepsilon} (x)\mathrm{e}^{-\varepsilon |x|} \right),\ \ \ x\in \Bb{R}.
\end{equation}

To prove $w_{\varepsilon} \in \mathcal{W}^{*} (\Bb{R})$, let $x_{0}\in \Bb{R}$, $\{x_{n}\}_{n \geq 1} \subset \Bb{R}$, $\lim_{n\to \infty} x_{n} =x_{0} $, $\Bb{N}_{1} := \{n~\geq~1 \ | \ |\theta_{\varepsilon} (x_{n})|\mathrm{e}^{-\varepsilon |x_{n}|} > \mathrm{e}^{-\varepsilon |x_{0}|}\, \}$ and let us choose
an infinite sequence $\Bb{N}_{2} :=\{n_{k}\}_{k\geq 1}\subset \Bb{N}$  such that $ \limsp_{n\to \infty} w_{\varepsilon} (x_{n}) = \lim_{k\to \infty} w_{\varepsilon} (x_{n_{k}})$. Since for every $n \in \Bb{N}\setminus \Bb{N}_{1}$ we have $|\theta_{\varepsilon} (x_{n})|\mathrm{e}^{-\varepsilon |x_{n}|} \leq \mathrm{e}^{-\varepsilon |x_{0}|}$,  \eqref{plem81zb} and \eqref{th1f1} yield  $w_{\varepsilon} (x_{n}) \leq w_{\varepsilon} (x_{0})$, $n  \in \Bb{N}\setminus \Bb{N}_{1}$. Thus, if
$\Bb{N}_{2} \cap (\Bb{N}\setminus \Bb{N}_{1})$ is infinite, then $ \limsp_{n\to \infty} w_{\varepsilon} (x_{n})\leq w_{\varepsilon} (x_{0})$. Otherwise, it suffices to consider the case $\Bb{N}_{2} \subset \Bb{N}_{1}$ in which
$\lim_{k\to \infty}|\theta_{\varepsilon} (x_{n_{k}})|  = 1$ and therefore $ \lim_{k\to \infty} w_{\varepsilon} (x_{n_{k}}) \leq \max \{
 \beta_{\varepsilon} (x_{0} - \mathrm{e}^{-\varepsilon |x_{0}|}),  \beta_{\varepsilon} (x_{0} + \mathrm{e}^{-\varepsilon |x_{0}|})\}\leq w_{\varepsilon} (x_{0})$,
 by virtue of \eqref{plem81zb} and the upper semi-continuity of  $\beta_{\varepsilon}$. This completes the proof of $w_{\varepsilon} \in \mathcal{W}^{*} (\Bb{R})$ (see \cite[Theorem 2, p.150]{nat1}).

Assume that $w_{\varepsilon} \notin \mathcal{W}^{{\rm{reg}}}(\Bb{R})$.  Then, for some $m \in \Bb{N}_{0}$
we have $(1+x^{2m})w_{\varepsilon} \notin \mathcal{W}^{{\rm{dens}}}(\Bb{R})$ and by Theorem~\ref{tha}
there exists an entire function $F \in \mathcal{E}_{0} (\Bb{R})$  such that
\begin{equation}\label{plem81}
 \sum_{\lambda \in \Lambda_{F}}\frac{1}{\left(1+\lambda^{2m}\right) w_{\varepsilon}(\lambda)\left|F^{\,\prime}(\lambda)\right|} < \infty.
\end{equation}
It follows from $w_{\varepsilon}\in \mathcal{W}^{*}(\Bb{R})$ that $\sum_{\lambda \in \Lambda_{F} } 1/|F^{\,\prime}(\lambda)| < \infty$
and therefore \eqref{8j} holds for $B=F$.

By Theorem~\ref{thc},
\begin{equation}\label{plem81z}
 \beta_{\varepsilon} \in \mathcal{W}^{{\rm{reg}}}(\Bb{R}).
\end{equation}
From \eqref{plem81} and  \eqref{plem81zb} we obtain
\begin{equation}\label{plem81a}
  \sum_{\lambda \in \Lambda_{F} } \frac{1}{\left(1 + \lambda^{2m}\right)\, \beta_{\varepsilon}(\lambda +
  \theta_{\varepsilon}(\lambda)\mathrm{e}^{-\varepsilon |\lambda|})\, \left|F^{\,\prime} (\lambda)\right|} < \infty.
\end{equation}
Applying Lemma~\ref{lm1} for $\delta = \varepsilon /2$, we find $T_{\varepsilon} > 0$
such that $\mathrm{e}^{-\varepsilon x/2} \leq \rho_{\varepsilon /2}$, $x \geq T_{\varepsilon}$,
and then we find an entire function $D \in \mathcal{E}_{0}(\Bb{R})$ with zeros $\Lambda_{D}  =
\left\{d_{\lambda}\right\}_{\lambda \in \Lambda_{B}}$, where
\begin{equation*}
  d_{\lambda} = \lambda,\ \ \  \lambda \in \Lambda_{F} \cap [-T_{\varepsilon}, T_{\varepsilon}],\ \ \
 d_{\lambda} = \lambda + \theta (\lambda) \mathrm{e}^{-\varepsilon |\lambda|},\ \ \
 \lambda \in \Lambda_{F} \setminus [-T_{\varepsilon}, T_{\varepsilon}].
\end{equation*}
Hence, in view of \eqref{11j} and \eqref{plem81a} we have
\begin{eqnarray*}
 \infty &>&   \sum_{\lambda \in \Lambda_{F} \setminus [-T_{\varepsilon}, T_{\varepsilon}]} \frac{1}{\left(1 + \lambda^{2m}\right)\,
 \beta_{\varepsilon} \left(\lambda + \theta_{\varepsilon}(\lambda)\mathrm{e}^{-\varepsilon |\lambda|}\right)\, \left|F^{\,\prime}(\lambda)\right|}
  \\ &  \geq&
\frac{1}{C_{\varepsilon/2}}\sum_{\lambda \in \Lambda_{F}  \setminus [-T_{\varepsilon}, T_{\varepsilon}]}\frac{1 + d_{\lambda}^{2m}}{1 +
\lambda^{2m}}\ \frac{1}{\left(1 + d_{\lambda}^{2m}\right)\, \beta_{\varepsilon}  \left(d_{\lambda} \right)\,  \left|D^{\,\prime} (d_{\lambda} )\right|}
\\ &  \geq&
\frac{1}{2^{2m} C_{\varepsilon/2}}\sum_{\lambda \in \Lambda_{F}  \setminus [-T_{\varepsilon}, T_{\varepsilon}]}
\frac{1}{\left(1 + d_{\lambda}^{2m}\right)\, \beta_{\varepsilon} \left(d_{\lambda} \right)\, \left|D^{\,\prime} (d_{\lambda}  )\right|} \ ,
\end{eqnarray*}
from which it follows that
\begin{equation*}
  \sum\limits_{\lambda \in \Lambda_{D}} \left(1 + \lambda^{2m}\right)^{-1} \beta_{\varepsilon} (\lambda )^{-1}
  \ \left|D^{\,\prime}(\lambda )\right|^{-1} < \infty \ .
\end{equation*}
By Theorem~\ref{tha} this means that $(1+ x^{2m})\cdot \beta_{\varepsilon}\notin \mathcal{W}^{{\rm{dens}}}(\Bb{R})$ and therefore
$\beta_{\varepsilon}\notin \mathcal{W}^{{\rm{reg}}}(\Bb{R})$. This contradicts \eqref{plem81z} and finishes the proof of Lemma~\ref{lm2}.
\end{proof}

We are now ready to prove our main result.

\begin{theorem}\label{th1}
For arbitrary $w\in \mathcal{W}^{{\rm{reg}}}(\Bb{R})$ and $\varepsilon > 0$ there exists $W_{\varepsilon} \in C^{\infty} (\Bb{R})$ such that
$W_{\varepsilon}\in \mathcal{W}^{{\rm{reg}}} (\Bb{R})$ and $W_{\varepsilon} (x) \geq w (x) + \mathrm{e}^{- \varepsilon |x|}$ for all $x\in \Bb{R}$.
\end{theorem}
\begin{proof} Since the statement of the theorem for $\varepsilon = \varepsilon_{0} > 0$ implies its validity for all
$\varepsilon \geq \varepsilon_{0}$, we can assume without loss of generality that $\varepsilon \in (0,1)$.

\medskip
Let $w_{\varepsilon}$ be defined as in \eqref{th1f1}, $\beta_{\varepsilon}$ as in \eqref{plem81za} and
\begin{equation*}
 \Omega_{\rho} (x) := \sup_{|s| \le \rho \mathrm{e}^{- \varepsilon |x|}} \beta_{\varepsilon} (x + s),\ \ \ x \in \Bb{R},\ \ \rho \in (0,1].
\end{equation*}
Since
\begin{equation}\label{mp2}
 w(x) \leq w(x) + \mathrm{e}^{-\varepsilon |x|} = \beta_{\varepsilon}(x) \leq \Omega_{\rho}(x) \leq w_{\varepsilon} (x),\ \  x \in \Bb{R},
\end{equation}
by Lemma~\ref{lm2},
\begin{equation*}
  \Omega_{\rho} \in \mathcal{W}^{{\rm{reg}}} (\Bb{R}),\ \ \rho \in (0,1].
\end{equation*}

\medskip
\noindent
For arbitrary $\omega \in C^{\infty} (\Bb{R})$ satisfying
\begin{equation}\label{mp1}
0 < \omega (x)    \leq  \mathrm{e}^{- \varepsilon | x | }/4, \ \ \  x \in \Bb{R},
\end{equation}
let us introduce
\begin{equation}\label{mp1x}
 K_{\omega} (x, t) :=\left(\int\nolimits_{-\omega (x)}^{\omega (x)}
  \mathrm{exp}\left(- \frac{\omega (x)^{2}}{\omega (x)^{2} - s^{2}}\right) \! {\mathrm{d}} s \!\right)^{-1}
  \! \mathrm{exp}\left(- \frac{\omega (x)^{2}}{\omega (x)^{2} - t^{2}} \right),
\end{equation}
where $t \in ( - \omega (x) , \omega (x) )$,  $x \in \Bb{R}$ and $K_{\omega} (x, \pm \omega (x)) := 0$. For example, we may take $\omega (x) =(1/4)\,
\mathrm{exp} (-x^{2}-  {\varepsilon^{2}}/{4})$.
Obviously,
\begin{equation*}
\int_{-\omega (x)}^{\omega (x)} K_{\omega} (x, t) \mathrm{d}t = 1, \ \ \  x \in \Bb{R},
\end{equation*}
and therefore the weight
\begin{equation}\label{mp1y}
  W_{\varepsilon} (x) :=   \int_{-\omega (x)}^{\omega (x)}  K_{\omega} (x, t) \Omega_{1/2} (x +t)
 \mathrm{d} t = \int_{x-\omega (x)}^{ x+\omega (x)}  K_{\omega} (x, t -x) \Omega_{1/2} (t) \mathrm{d}t
\end{equation}
belongs to $ C^{\infty}(\Bb{R})$.

\medskip
Let $x\in \Bb{R}$ be arbitrary and let $t \in \Bb{R}$ satisfy $|t|\leq \omega (x)$. Then, by \eqref{mp1} we have
$|t|\leq \mathrm{e}^{- \varepsilon |x|}/4$ and the inequalities $\mathrm{e}^{1/4}\leq 4/3$ and $0 < \varepsilon< 1$ imply
$(3/4)\mathrm{e}^{-\varepsilon|x|} \leq \mathrm{e}^{-\varepsilon |x+t|} \leq (4/3)\mathrm{e}^{-\varepsilon|x|}$. Thus, for every
$\rho \in (1/3, 1]$,
\begin{equation*}
  {(3\rho -1)}\mathrm{e}^{-\varepsilon|x|}/{4} \leq \rho \mathrm{e}^{-\varepsilon |x+t|} + t \leq(16\rho +3)\mathrm{e}^{-\varepsilon|x|}/12,
\end{equation*}
and therefore
\begin{equation*}
   \Omega_{(3\rho -1)/4} (x )\! \leq \! \Omega_{\rho} (x + t)
 \! \leq \!  \Omega_{(16\rho +3)/12} (x ),\ \  \rho \!  \in \!  (1/3, 1),\ \  |t|\!\leq\! \mathrm{e}^{- \varepsilon | x | }/4,\ \  x \in \Bb{R},
\end{equation*}

\noindent
from which we infer for $\rho = 1/2$ that
\begin{gather}\hspace{-0.2cm}
\beta_{\varepsilon} (x)\! \leq\! \Omega_{1/8} (x)\! \leq\! \Omega_{1/2} (x + t) \!\leq\! \Omega_{11/12} (x)\! \leq\!
\Omega_{1} (x)\! =\! w_{\varepsilon} (x),\, |t|\!\leq\!\omega (x), \, x \in \Bb{R}.
\label{mp4x}\end{gather}
In view of \eqref{mp2} this means that the weight $W_{\varepsilon}$ satisfies
\begin{equation}\label{mp4}
  w (x) + \mathrm{e}^{-\varepsilon |x|} \leq  W_{\varepsilon} (x)  \leq w_{\varepsilon} (x),\ \  x \in \Bb{R}.
\end{equation}
It follows from the right-hand side inequality of \eqref{mp4} that $ W_{\varepsilon}\in \mathcal{W}^{{\rm{reg}}} (\Bb{R} )$
and therefore the left-hand side inequality of  \eqref{mp4} completes the proof.
\end{proof}

Since the weight $W_{\varepsilon}$ defined in \eqref{mp1y} depends on an arbitrary
function  $\omega \in C^{\infty} (\Bb{R})$ satisfying \eqref{mp1}, we prove in the next corollary
that the special choice $\omega = \phi_{\varepsilon}$ yields a good upper estimate for $ W_{\varepsilon}^{\, \prime}$.
Here,
\begin{align}\label{fz0cor1}&
 \phi_{\varepsilon} (x) := \frac{\mathrm{e}^{-\varepsilon }}{4 \kappa} \int_{-1}^{1} {\rm{e}}^{- \tfrac{1}{1-t^{2}}}  \mathrm{e}^{- \varepsilon | x +t | } {\rm{d}} t, \ \ x \in \Bb{R} , \ \ \varepsilon > 0 \ , \\ \label{fz1cor1}&
  \kappa := \int\nolimits_{-1}^{1}
  {\rm{e}}^{- \tfrac{1}{1-t^{2}}} \mathrm{d} t = \frac{K_{1} (1/2) - K_{0} (1/2)}{\sqrt{ \mathrm{e}}} \in \left(\frac{1.2}{  \mathrm{e} } \ , \frac{1.21}{  \mathrm{e} }\right), \end{align}

\noindent
$K_{1}$, $K_{0}$ are modified Bessel functions (see \cite[(13), p.5]{erd})  and \eqref{fz1cor1} is proved in Section~\ref{corollary1}.

\begin{corollary}\label{cor1} Let $\varepsilon \in (0, 1)$,  $w\in \mathcal{W}^{{\rm{reg}}}(\Bb{R})$ and $w_{\varepsilon}$ be defined as in \eqref{th1f1}.
 Then there exists a weight $W_{\varepsilon} \in C^{\infty} (\Bb{R}) \cap \mathcal{W}^{{\rm{reg}}} (\Bb{R})$ such that $W_{\varepsilon} (x) \geq w (x) + \mathrm{e}^{- \varepsilon |x|}$ and $| W_{\varepsilon}^{\, \prime} (x) | \leq 74\, \mathrm{e}^{\varepsilon |x| } w_{\varepsilon} (x)$ for all $x\in \Bb{R}$.
\end{corollary}

Theorem~\ref{th1} allows to assume without loss of generality that each weight in the regular part of Bernstein's approximation problem is
continuous and positive on the whole real axis. It also allows to apply for this part of the problem the sufficient
conditions for the denseness of algebraic polynomials in $C^{\,0}_{w}(\Bb{R})$ obtained earlier under this assumption (see \cite[p.869]{pol},
\cite[p.80]{m}). On the other hand, Lemma~\ref{lm2} makes it possible to replace any weight
$w \in\mathcal{W}^{*}(\Bb{R})$ by the greater step function
\begin{eqnarray*}
  \widehat{w} (x) &=& \sum_{n\in \Bb{Z}} w_{{\fo{n}}} \chi_{{\fo{[\sigma_{n}\log (1+|n|), \sigma_{n}\log (1+|n +1|))}}}(x),\ \ \ x\in \Bb{R},
 \\ w_{n} &:=& \sup_{{\fo{x \in [\sigma_{n}\log (1+|n|), \sigma_{n}\log (1+|n +1|)]}}} w (x),\ \ \sigma_{n} := \sign (n),\ \  n\in \Bb{Z},
\end{eqnarray*}
such that algebraic polynomials are regularly dense in $C^{\,0}_{w}(\Bb{R})$ if and only if they are regularly dense in  $C^{\,0}_{\widehat{w}}(\Bb{R})$.
Here, $\sign (n)$ is equal to $1$ if $n > 0$, $0$ if $n=0$ and $-1$ if $n < 0$.

Notice also that Theorem~\ref{th1} can be efficiently applied to a representation of the so-called $p$-regular measures for $1 \leq p < \infty $.
Recall (see \cite[p.250]{bs}) that a non-negative Borel measure $\mu $ on $\Bb{R}$ is  called $p$-{\emph{regular}} if   all its moments $\int_{\Bb{R}} x^{n}{\rm{ d}}\mu (x)$, $n \geq 0$, are finite and  algebraic polynomials are dense in $L_p (\Bb{R}, (1+x^{2})^{n p} {\rm{d}} \mu (x))$ for every $n \geq 0 $.
Here, for arbitrary non-negative Borel measures $\mu$, $\nu$ on $\Bb{R}$ and $g \in L_1 (\Bb{R}, {\rm{d}} \mu)$, we write ${\rm{d}} \nu (x) = g (x) {\rm{d}} \mu (x) $ or
${\rm{d}} \nu = g  {\rm{d}} \mu $ if $\nu (A) = \int_{A} g (x) {\rm{d}} \mu (x)$ for arbitrary Borel subset $A$ of $\Bb{R}$.
According to
\cite[Lemma 4, p.203]{b1}, if $\mu$ is  $p$-regular, then there exists  a finite non-negative Borel measure $\nu $ on $\Bb{R}$ and $w \in   \mathcal{W}^{{\rm{reg}}} (\Bb{R})$ such that ${\rm{d}} \mu = w^{p} \,  {\rm{d}} \nu$ (the converse is evident). Taking for this $w$ the weight $W_{\varepsilon}$ from Theorem~\ref{th1},
we obtain ${\rm{d}} \mu = w^{p}  \, {\rm{d}} \nu = W_{\varepsilon}^{p} (w/W_{\varepsilon})^{p} \,  {\rm{d}} \nu = W_{\varepsilon}^{p} \,  {\rm{d}} \widetilde{\nu}$ where $\widetilde{\nu}$ is also a non-negative finite Borel measure on $\Bb{R}$ as follows from  ${\rm{d}} \widetilde{\nu} = (w/W_{\varepsilon})^{p} \,  {\rm{d}} \nu$ and $w (x) \leq  W_{\varepsilon} (x)$ for all $x \in \Bb{R}$. Thus, the following assertion holds.

\begin{corollary}\label{cor2} Let $1 \leq p < \infty $ and a measure $\mu$ is $p$-regular. Then, for every $\varepsilon > 0$, there exists  a finite non-negative Borel measure $\nu_{\varepsilon} $ on $\Bb{R}$ and a weight $W_{\varepsilon} \in C^{\infty} (\Bb{R}) \cap \mathcal{W}^{{\rm{reg}}} (\Bb{R})$ such that $W_{\varepsilon} (x) \geq \mathrm{e}^{- \varepsilon |x|}$ for all $x\in \Bb{R}$ and ${\rm{d}} \mu \,  = \,  W_{\varepsilon}^{p} \ {\rm{d}} \nu_{\varepsilon}$.
\end{corollary}

\vspace{0.5cm}
\section{Auxiliary Results}\label{aux}

\vspace{0.25cm}
\begin{lemma}\label{lem1}
 Let the real numbers $a $, $b $,  $x $  and $\Delta \in (0,1)$ satisfy
\begin{equation}\label{lem11}
 b \in (a-\Delta^{2}, a+\Delta^{2}), \ \ \ 0\notin (a-\Delta , a+\Delta ) \ \ \mbox{and}\ \ \ x\notin (a-2\Delta, a+2\Delta).
\end{equation}
Then,
\begin{equation*}
     \left| \left(1 - \frac{x}{a}\right) \left(1 - \frac{x}{b}\right)^{-1} \right| \leq \left(1 + \Delta\right)^{2}.
\end{equation*}
\end{lemma}

\begin{proof}
The conditions \eqref{lem11} imply $|a|\geq \Delta$, $ b \in  (a-\Delta , a+\Delta )$ and therefore
$|x-b|\geq \Delta$.  Thus,  $ | |b| - |a| | \leq |b -  a | \leq \Delta^{2}$ and
$ |b| \leq |a| + \Delta^{2}$, i.e. $ |b| /|a| \leq 1 + \Delta^{2} / |a|  \leq 1 + \Delta$. Finally,
\begin{align*}
    \left| \frac{1 - {x}/{a}}{1 - {x}/{b}} \right| & =
   \frac{\left|b\right|}{\left|a\right|} \frac{|x-a|}{|x-b|}  = \frac{\left|b\right|}{\left|a\right|}
\frac{\left|b - a + (x - b)\right|}{\left|x - b\right|}  \\ & \leq \frac{\left|b\right|}{\left|a\right|}
\frac{\left|b - a\right| + \left| x - b\right|}{\left|x - b\right|} \leq
\left(1 +  \Delta \right) \cdot \left(1 + \frac{\left|b - a\right|}{\left|x - b\right|}\right) \leq \left(1 + \Delta \right)^{2},
\end{align*}
which completes the proof.
\end{proof}

\begin{lemma}\label{lem2}
Let $\varepsilon \in (0, 1/(2 \mathrm{e}))$, $C_{\varepsilon} \in (0, +\infty )$ and $f$ be an entire function satisfying
\begin{equation}\label{lem21}
   |f (z)| \leq C_{\varepsilon} \mathrm{e}^{{\fo{\varepsilon |z|}}},\ \ \ z \in \Bb{C} \ .
\end{equation}
Then,
\begin{equation*}
 |f^{\,\prime} (z)|,\ \ \ \left| \frac{f (z)}{z -\lambda}\right|\leq C_{\varepsilon}
 \mathrm{e}^{{\fo{\varepsilon |z|}}},\ \ \lambda \in \Lambda_{f},\ \ z \in \Bb{C} \ .
\end{equation*}
\end{lemma}

\begin{proof}
Cauchy's formula  \cite[(3), p.~81]{tit1}
\begin{equation*}
  f^{\, \prime} (z) = \frac{1}{2\pi \mathrm{i}} \int_{|z-\zeta|=1/\varepsilon} \frac{f(\zeta) \mathrm{d}\zeta}{(\zeta -z)^{2}}
\end{equation*}
and \eqref{lem21} for any $z \in \Bb{C}$ yield
\begin{align*}
  \left| f^{\prime} (z)\right| & \leq \varepsilon
  \max_{|\zeta -z|=1/\varepsilon} \left|f(\zeta)\right| \leq  \varepsilon  C_{\varepsilon}\max_{|\zeta -z|=1/\varepsilon}
  \mathrm{e}^{{\fo{\varepsilon |\zeta|}}} \\ & \leq
 \varepsilon  C_{\varepsilon}  \mathrm{e}^{{\fo{\varepsilon (|z| + 1/\varepsilon)}}}=  \varepsilon \mathrm{e} C_{\varepsilon}
  \mathrm{e}^{{\fo{\varepsilon |z|}}} \leq C_{\varepsilon} \mathrm{e}^{{\fo{\varepsilon |z|}}}.
\end{align*}

For arbitrary $\lambda \in \Lambda_{f}$ and  $z \in \Bb{C}$
satisfying $ |z-\lambda|\geq 1/(2 \varepsilon)$
 it follows from \eqref{lem21} that
\begin{equation*}
  \left|  \frac{f (z)}{z-\lambda}\right| \leq 2 \varepsilon C_{\varepsilon} \mathrm{e}^{{\fo{\varepsilon |z|}}} \leq C_{\varepsilon} \mathrm{e}^{{\fo{\varepsilon |z|}}},
\end{equation*}
which by the maximum modulus principle \cite[p.~165]{tit1} yields
\begin{align*}
   \left|  \frac{f (z)}{z-\lambda}\right| & \leq
    \max_{|\zeta - \lambda| = 1/(2 \varepsilon)}
    \left|  \frac{f (\zeta)}{\zeta-\lambda}\right| = 2 \varepsilon
     \max_{|\zeta - \lambda|  = 1/(2 \varepsilon)} |f (\zeta)| \\ &  \leq 2 \varepsilon C_{\varepsilon} \max_{|\zeta - \lambda|  =
        1/(2 \varepsilon)}\mathrm{e}^{{\fo{\varepsilon |\zeta|}}}
          \leq 2 \varepsilon C_{\varepsilon} \mathrm{e}^{{\fo{\varepsilon (|z| +1/\varepsilon)}}} \leq  2 \varepsilon \mathrm{e} C_{\varepsilon}
  \mathrm{e}^{{\fo{\varepsilon |z|}}} \leq C_{\varepsilon}
  \mathrm{e}^{{\fo{\varepsilon |z|}}},
\end{align*}
provided that $|z-\lambda|\leq 1/(2 \varepsilon)$. This finishes the proof of Lemma~\ref{lem2}.
\end{proof}

\begin{lemma}\label{lem3}
Let $\varepsilon \in (0, 1/(2 \mathrm{e}))$, $C_{\varepsilon}  \in (0, +\infty )$  and $B$ be an
entire function from the class ${\mathcal{E}}_{0} (\Bb{R})$ satisfying
\begin{equation}\label{lem31}
 {\rm{(a)}} \ \   |B (z)| \leq C_{\varepsilon} \mathrm{e}^{{\fo{\varepsilon |z|}}},\ \  z \in \Bb{C},\ \ \ \ {\rm{(b)}} \ \
 \Theta_{B} := \sum_{\lambda \in \Lambda_{B}} \frac{1}{ \left|B^{\,\prime } (\lambda)\right|} < \infty \ .
\end{equation}
Then, for arbitrary $\lambda \in \Lambda_{B}$ the inequality
\begin{equation}\label{lem32}
   \left| \frac{B (x)}{x-\lambda}\right| \geq \left| B^{\,\prime} ( \lambda)\right| / 2
\end{equation}
holds for every real $x$ satisfying
\begin{equation}\label{lem33}
 |x-\lambda|\leq \frac{\mathrm{e}^{{\fo{-\varepsilon}} }}{1+ 2 C_{\varepsilon} \Theta_{B}} \mathrm{e}^{{\fo{- \varepsilon |\lambda|}}}.
\end{equation}
Thus,
\begin{equation}\label{lem34}
   \min_{\mu \in \Lambda_{B}\setminus \{\lambda\}} \left|\lambda - \mu \right| > \frac{\mathrm{e}^{{\fo{-\varepsilon}} }}{
  1+ 2 C_{\varepsilon}  \Theta_{B}} \mathrm{e}^{{\fo{- \varepsilon |\lambda|}}}, \  \ \   \lambda \in \Lambda_{B}.
\end{equation}
\end{lemma}

\begin{proof}
Let $\lambda \in \Lambda_{B}$ and
\begin{equation*}
   B_{\lambda} ( x) := \frac{B (x)}{x-\lambda} \ .
\end{equation*}
Obviously, $ B_{\lambda} ( \lambda) =B^{\,\prime}(\lambda)$ and it follows from Lemma~\ref{lem2} that
\begin{equation*}
   |B_{\lambda}^{\,\prime} (z)| \leq C_{\varepsilon} \mathrm{e}^{{\fo{\varepsilon |z|}}},\ \ \ z \in \Bb{C}.
\end{equation*}
Furthermore,  \eqref{lem31}(b) yields
\begin{equation*}
  \left|B^{\,\prime}(\lambda)\right| \geq  \Theta_{B}^{-1}.
\end{equation*}

Assume that  $x \in [-1,1]$ and
\begin{equation*}
    \left|B_{\lambda} (x + \lambda) - B_{\lambda} ( \lambda)\right| >\left| B_{\lambda} ( \lambda)\right| / 2.
\end{equation*}
Then,
\begin{eqnarray*}
 \Theta_{B}^{-1}/2  &\leq & \left|B^{\,\prime}(\lambda)\right|/2 = \left| B_{\lambda}(\lambda)\right| / 2
   <  \left|B_{\lambda} (x + \lambda)-     B_{\lambda} ( \lambda)\right| \\
   &  = & \left| \int_{0}^{|x|}B_{\lambda}^{\,\prime} (\lambda + \sigma t) \mathrm{d}t \right|\leq C_{\varepsilon}\int_{0}^{|x|}
   \mathrm{e}^{{\fo{\varepsilon |\lambda + \sigma t|}}} \mathrm{d} t
    \leq  C_{\varepsilon}  \mathrm{e}^{{\fo{\varepsilon (1+|\lambda|)}}} |x| \\
    &   < & \left( C_{\varepsilon} + \Theta_{B}^{-1}/2 \right) \mathrm{e}^{{\fo{\varepsilon (1+|\lambda|)}}} |x|,
\end{eqnarray*}
where $\sigma =1$ if $x>0$ and $\sigma =-1$ if $x<0$. This means that if
\begin{equation*}
    |x| \leq  \frac{\mathrm{e}^{{\fo{-\varepsilon (1+|\lambda|)}}}}{1+ 2 C_{\varepsilon}  \Theta_{B}} \ ,
\end{equation*}
then
\begin{equation*}
   \left|B_{\lambda} (x + \lambda)-B^{\,\prime} ( \lambda)\right| \leq \left| B^{\,\prime}(\lambda)\right| / 2,
\end{equation*}
and therefore
\begin{eqnarray*}
   | B_{\lambda } (\lambda +x)  | & =& \left| B^{\,\prime} ( \lambda) + B_{\lambda } (\lambda +x) - B^{\,\prime} ( \lambda) \right| \\
  &\geq& \left|B^{\,\prime}(\lambda)\right| - | B_{\lambda } (\lambda +x) - B^{\,\prime}(\lambda) |\geq \left| B^{\,\prime}(\lambda)\right| / 2,
\end{eqnarray*}
which was to be proved.
\end{proof}

\begin{lemma}\label{lem4}
 Let $f: \Bb{R}\to\Bb{R}$ be integrable on every compact segment $[a,b]$ of the real line, $\omega \in C^{\infty}(\Bb{R})$ be strictly positive on $\Bb{R}$,
$ \kappa $ be defined in \eqref{fz1cor1}  and
  \begin{gather}\label{f1lem4}
  f_{\omega} (x)  := \frac{1}{\kappa}
\int_{-1}^{ 1}  f (x+ t \omega (x)) \ \mathrm{exp}\left(- \frac{1}{1 -t^{2}} \right) \mathrm{d} t \ ,  \ x \in \Bb{R} \ .
  \end{gather}
\noindent
Then,  $f_{\omega}   \in C^{\infty}(\Bb{R})$  and for every $ x \in \Bb{R} $ we have
\begin{align}\nonumber
   f_{\omega}^{\,\prime} (x) & =   \frac{1}{   \kappa \, \omega (x)}
\int\limits_{-1}^{ 1}  f (x+t\omega (x))  \frac{2 t }{\left(t^{2} -1\right)^{2}} \ \mathrm{exp}\left(- \frac{1}{1 -t^{2}}  \right)  {\rm{d}} t \\ &
-   \frac{\omega^{\,\prime} (x)}{\omega (x)}  f_{\omega} (x)  + \frac{2 \, \omega^{\,\prime} (x)}{ \kappa \,  \omega (x)}  \int\limits_{-1}^{1}  f (x+t\omega (x))   \frac{
 t^{2}}{\left(t^{2} -1\right)^{2}}\ \mathrm{exp}\left(- \frac{1}{1 -t^{2}}  \right)   {\rm{d}} t  .
\label{f2lem4}\end{align}

\end{lemma}
\begin{proof}
Since
\begin{gather*}
\kappa  f_{\omega} (x)  =
\frac{1}{ \omega (x)}
\int_{x-\omega (x)}^{ x+\omega (x)}  f (t) \, \mathrm{exp}\left(- \frac{\omega (x)^{2}}{\omega (x)^{2} -(t-x)^{2}} \right) \mathrm{d} t  , \ \ x \in \Bb{R} ,
\end{gather*}

\noindent
then $f_{\omega}   \in C^{\infty}(\Bb{R})$  and for arbitrary $ x \in \Bb{R} $ we obtain
\begin{gather*}
 \kappa  f_{\omega}^{\,\prime} (x)    = -\kappa   f_{\omega} (x)   \frac{\omega^{\,\prime} (x)}{\omega (x)}+\frac{1}{ \omega (x)}
\int_{x-\omega (x)}^{ x+\omega (x)} f (t)  T_{\omega} (x, t) \mathrm{d} t ,
\end{gather*}

\noindent
where
\begin{multline*}
    T_{\omega} (x, t) :=  \frac{\mathrm{d}}{\mathrm{d} x}  \mathrm{exp}\left(- \frac{\omega (x)^{2}}{\omega (x)^{2} -(t-x)^{2}} \right)\\  =
   \left[
\frac{ 2 \, \omega (x) \, \omega^{\,\prime} (x)(t-x)^{2}}{\left((t-x)^{2} -\omega (x)^{2}\right)^{2}} + \frac{2 \, \omega (x)^{2}  (t-x)}{\left((t-x)^{2} -\omega (x)^{2}\right)^{2}} \right]  \mathrm{exp}\left(\!- \frac{\omega (x)^{2}}{\omega (x)^{2} -(t-x)^{2}} \right) ,
\end{multline*}

\noindent
from which \eqref{f2lem4} follows easily by the change of variables. Lemma~\ref{lem4} is proved.
\end{proof}

\vspace{0.15cm}
\section{\texorpdfstring{Proof of Lemma~\ref{lm1}}{Proof of Lemma 1} }\label{lemma1}

\vspace{0.25cm}
\subsection{}
If Lemma~\ref{lm1} is proved for $\delta = \delta_{0} > 0$, then for arbitrary $\delta_{1}>\delta_{0}$ it follows from
$|\lambda - d_{\lambda} | \leq \rho_{\delta_{0}} \, \mathrm{e}^{{{- \delta_{1} |\lambda|}}}\leq \rho_{\delta_{0}} \,
\mathrm{e}^{{{- \delta_{0} |\lambda|}}}$, $\lambda \in \Lambda_{B}$  that Lemma~\ref{lm1} also holds
for $\delta = \delta_{1}$ with $C_{\delta_{1}} = C_{\delta_{0}}$ and $\rho_{\delta_{1}} = \rho_{\delta_{0}}$.
Therefore, it is sufficient to prove Lemma~\ref{lm1} only for those numbers $\delta$ which satisfy
\begin{equation*}
  0 < \delta < {1}/{\mathrm{e}}.
\end{equation*}

\vspace{0.15cm}
\subsection{}\label{explicit} Let $B$ be an entire function satisfying the conditions of Lemma~\ref{lm1}. Then, these conditions are met
by any translation of $B$ of the form $B_{T_{a}} (z) := B (z+a)$, $a\in \Bb{R}\setminus\{0\}$ because
$\Lambda_{B_{T_{a}}} = \Lambda_{B}-a$, $\Theta_{B_{T_{a}}}=  \Theta_{B}$ and $B_{T_{a}} \in {\mathcal{E}}_{0}(\Bb{R})$,
where $\Theta_{B} $ denotes the value of the series in \eqref{8j}.

We show that if Lemma~\ref{lm1} is proved for the function $B$ then it also holds for any $B_{T_{a}}$, $a\in \Bb{R}\setminus\{0\}$,
with constants $ \rho_{\delta} (B_{T_{a}}) = \mathrm{e}^{-\delta |a|}\rho_{\delta} (B)$ and  $ C_{\delta} (B_{T_{a}}) = C_{\delta} (B)$.

\smallskip
Let $\delta > 0$,  $a$ be an arbitrary nonzero real number and $E:= B_{T_{a}}$. If $\left\{e_{\lambda}\right\}_{\lambda \in \Lambda_{E} }$
is any collection of real numbers satisfying $|\lambda - e_{\lambda}| \leq \rho_{\delta} (E) \exp{(- \delta |\lambda|)}$, $\lambda \in \Lambda_{E}$,
then in view of $\Lambda_{E} = \Lambda_{B}-a$ we have
\begin{equation*}
  |\lambda -a - e_{\lambda - a}| \leq \rho_{\delta} (E) \mathrm{e}^{- \delta |\lambda -a |} \leq \mathrm{e}^{\delta |a|} \rho_{\delta} (E)
  \mathrm{e}^{- \delta |\lambda  |} = \rho_{\delta} (B) \mathrm{e}^{- \delta |\lambda  |},\ \  \lambda \in \Lambda_{B},
\end{equation*}
and therefore the numbers $d_{\lambda} := e_{\lambda - a} +a$, $\lambda \in \Lambda_{B}$, satisfy condition \eqref{9j}.
Thus, there exists an entire function $D \in \mathcal{E}_{0} (\Bb{R})$ such that $ \Lambda_{D}  = \left\{d_{\lambda}\right\}_{\lambda
\in \Lambda_{B} }$ and $|B^{\,\prime}(\lambda)| \leq C_{\delta} (B) \,|D^{\,\prime}(d_{\lambda})|$, $\lambda \in \Lambda_{B}$.
Then, for the function $G (z):= D (z+a)$ we have $G \in \mathcal{E}_{0} (\Bb{R})$, $\Lambda_{G}  = \Lambda_{D}-a  =
\left\{d_{\lambda}-a\right\}_{\lambda \in \Lambda_{B}  } = \left\{d_{\lambda+a}-a\right\}_{\lambda \in \Lambda_{E}  } =
\left\{e_{\lambda}\right\}_{\lambda \in \Lambda_{E}}$ and $|E^{\,\prime} (\lambda)|=|B^{\,\prime} (a+\lambda)| \leq
C_{\delta} (B) \,|D^{\,\prime}(d_{a+\lambda})|=C_{\delta} (B) \,|G^{\,\prime }(d_{a+\lambda} -a)| = C_{\delta} (B) \,|G^{\,\prime }(e_{\lambda})|$,
$\lambda\in \Lambda_{E}$. This implies the validity of  Lemma~\ref{lm1} for $B_{T_{a}}$, as claimed.

\smallskip
We conclude that to prove Lemma~\ref{lm1} for all translations $B_{T_{a}}$, $a\in \Bb{R}$, of the entire function $B$
it is sufficient to prove it for at least one of them. We specify the translation of $B$ by choosing an $a\in \Bb{R}\setminus \Lambda_{B}$
such that $\min_{\lambda\in\Lambda_{B}, \lambda > a} (\lambda - a) = \min_{\lambda\in\Lambda_{B}, \lambda < a}(a- \lambda)$
if $\Lambda_{B}$ is unbounded in both directions, $ a > 1+ \max \Lambda_{B}$
if $\Lambda_{B}$ is bounded from above and $ a < - 1+ \min\Lambda_{B}$ if $\Lambda_{B}$ is bounded from below.
Considering such $B_{T_{a}}$ as the initial function $B$ in Lemma~\ref{lm1}, we can therefore assume that the set
$\Lambda_{B}$ of all zeros of $B$ in Lemma~\ref{lm1} obeys the following additional properties:

\begin{equation}\label{plm3a}\hspace{-0.7cm}
\begin{array}{lll} & (a)\ \ \ \ \ \  0 \notin \Lambda_{B} ; &  \\[0.15cm]
& (b) \ \ \min\limits_{\lambda \in \Lambda_{B}, \lambda > 0} |\lambda| = \min\limits_{\lambda \in \Lambda_{B}, \lambda < 0} |\lambda|
& \mbox{if} \ \sup \Lambda_{B}=+\infty \ \mbox{and}\ \inf \Lambda_{B}= -\infty;\\[0.25cm]
& (c)\ \ \ \ \  \min \Lambda_{B} > 1 & \mbox{if} \  \inf \Lambda_{B} > -\infty;\\[0.15cm]
& (d)\ \ \  \ \  \max \Lambda_{B} < - 1 & \mbox{if} \ \sup \Lambda_{B} < +\infty.
\end{array}
\end{equation}
Observe that \eqref{plm3a}(b) means the existence of two neighboring zeros $\lambda_{1}, \lambda_{2} \in \Lambda_{B}$ of
$B$ (i.e., $\lambda_{1}<\lambda_{2}$,  $(\lambda_{1}, \lambda_{2}) \cap \Lambda_{B} = \emptyset$) such that $\lambda_{1} = - \lambda_{2}$.

\vspace{0.25cm}
\subsection{}
Denote by $\Theta_{B} $ the value of the series  in \eqref{8j} and let
\begin{equation}\label{plm2}
 \varepsilon := \delta/2 \in (0,\,  1/(2 \mathrm{e})\,),\  \
 \rho_{\delta}:= \left(\frac{\mathrm{e}^{{\fo{-\varepsilon }}}}{4+ 8C_{\varepsilon}\Theta_{B}} \right)^{2}\in (0,1/16),
\end{equation}
where
\begin{equation*}
  C_{\varepsilon} := \sup_{z \in \Bb{C}} \mathrm{e}^{{\fo{-\varepsilon |z|}}}  |B (z)| < \infty.
\end{equation*}
Then, for the function $B$ the conditions of Lemma~\ref{lem3} are fulfilled and \eqref{lem33} implies that
\begin{equation}\label{plm4}
  \left[\lambda_{1}\! -\! 2 \Delta_{\lambda_{1}},  \,  \lambda_{1}\! +\! 2 \Delta_{\lambda_{1}} \right] \cap
  \left[\lambda_{2}\! -\! 2 \Delta_{\lambda_{2}}, \, \lambda_{2}\! + \! 2 \Delta_{\lambda_{2}} \right] \!=
  \! \emptyset, \, \ \lambda_{1}, \lambda_{2} \!\in \!\Lambda_{B},\ \,  \lambda_{1} \!\neq \! \lambda_{2},
\end{equation}
where
\begin{equation}\label{plm5}
  \Delta_{\lambda}:=  \sqrt{\rho_{\delta}} \ \mathrm{e}^{{\fo{- \varepsilon |\lambda|}}}\in (0,1/4),\ \  \  \lambda\in \Lambda_{B}.
\end{equation}
Actually, assume that there exist $\lambda_{1}, \lambda_{2} \in \Lambda_{B}$ such that $\lambda_{1}<\lambda_{2}$,
$(\lambda_{1}, \lambda_{2}) \cap \Lambda_{B} = \emptyset$ and
\begin{equation}\label{plm6}
\lambda_{1} + 2 \Delta_{\lambda_{1}} \geq \lambda_{2}- 2 \Delta_{\lambda_{2}}.
\end{equation}
By virtue of \eqref{lem34},
\begin{equation}\label{plm7}
  \lambda_{1} < \lambda_{2} -  4 \Delta_{\lambda_{2}},\ \ \lambda_{1} + 4 \Delta_{\lambda_{1}} < \lambda_{2},
\end{equation}
and therefore
\begin{equation*}
  \lambda_{2} - \lambda_{1} > 2 \Delta_{\lambda_{1}} + 2 \Delta_{\lambda_{2}},
\end{equation*}
which contradicts \eqref{plm6} and proves \eqref{plm4}.

\smallskip
Introduce the following neighborhood of $\Lambda_{B}$:
\begin{equation}\label{plm15a}
  \Lambda_{B}^{\!\Delta} := \bigsqcup\nolimits_{\lambda \in \Lambda_{B}}
 \left[\, \lambda\! -\! 2 \Delta_{\lambda} ,  \,  \lambda\! +\! 2 \Delta_{\lambda} \, \right].
\end{equation}

We now prove that for any two neighboring zeros $\lambda_{1}<\lambda_{2}$ of $B$ the midpoint of the interval
$[\lambda_{1}, \lambda_{2}]$ does not belong to $\Lambda_{B}^{\!\Delta}$. In fact, it follows from
$\lambda_{1}, \lambda_{2} \in \Lambda_{B}$, $\lambda_{1}<\lambda_{2}$,
$(\lambda_{1}, \lambda_{2}) \cap  \Lambda_{B} = \emptyset$ and \eqref{plm7} that
\begin{equation*}
  \frac{\lambda_{1} + \lambda_{2}}{2} <  \lambda_{2} - 2\Delta_{\lambda_{2}} \, ,\ \ \
 \lambda_{1} + 2\Delta_{\lambda_{1}} <  \frac{\lambda_{1} + \lambda_{2}}{2} \ ,
\end{equation*}
which proves
\begin{equation}\label{plm7b}\lambda_{1}<\lambda_{2}\ , \  \lambda_{1}, \lambda_{2} \in \Lambda_{B},\ \
(\lambda_{1}, \lambda_{2}) \cap  \Lambda_{B} = \emptyset  \ \Rightarrow \
  \frac{\lambda_{1} + \lambda_{2}}{2} \notin \Lambda_{B}^{\!\Delta} \ .
\end{equation}
Together with \eqref{plm3a} this property means that
\begin{equation}\label{plm7c}
   0 \notin \Lambda_{B}^{\!\Delta} \ .
\end{equation}
Actually, if $\Lambda_{B}$ is unbounded in both directions, then according to \eqref{plm3a}(b) the origin is
the midpoint of a segment joining two neighboring zeros of $B$ which have opposite signs.
It follows from \eqref{plm7b} that \eqref{plm7c} holds. In the case when $\Lambda_{B}$ is bounded from one side
the distance $\min_{\lambda\in\Lambda_{B}} |\lambda|$ between
$0$ and $\Lambda_{B}$ is greater than $1$, by virtue of \eqref{plm3a}(c), (d). But in view of \eqref{plm5},
$2 \Delta_{\lambda} < 1/2$ and therefore \eqref{plm7c} follows readily from \eqref{plm15a}.

\vspace{0.25cm}
\subsection{}  If $\{d_{\lambda}\}_{\lambda\in \Lambda_{B}}$ are arbitrary numbers satisfying \eqref{9j},
it follows from \eqref{9j}, \eqref{plm2} and \eqref{plm5} that
\begin{equation}\label{plm9}
    d_{\lambda} \in  \left[\, \lambda\! -\! \Delta_{\lambda}^{2}  , \,   \lambda\! +\! \Delta_{\lambda}^{2} \, \right]
  \subset  \left[\, \lambda \!- \! \Delta_{\lambda} , \, \lambda \!+\! \Delta_{\lambda} \, \right],\ \ \  \lambda\in \Lambda_{B},
\end{equation}
and in view of \eqref{plm4},
\begin{equation}\label{plm10}
    d_{\lambda_{0}} \notin \left[\, \lambda \!- \! 2\Delta_{\lambda}  , \,  \lambda \!+\! 2\Delta_{\lambda} \, \right],\ \ \
\lambda_{0}  , \lambda \!\in \!\Lambda_{B},\ \ \,
\lambda_{0} \!\neq \! \lambda.
\end{equation}

It is worth remembering that according to the Lindel\"{o}f theorem \cite[Th.~15, p.~28]{lev2} a set
$\Lambda \subset \Bb{R}\setminus \{0\}$ is the set of all zeros of some entire function from the class
${\mathcal{E}}_{0} (\Bb{R})$ if and only if there exists a finite limit  of $\delta_{\Lambda} (R)$ and
${ n_{\Lambda} (R)}/{R} \to 0$ as $R \to +\infty$. Here,
\begin{equation*}\label{plm15}
 \delta_{\Lambda} (R) :=  \sum\nolimits_{\lambda \in \Lambda \cap (-R , R ) } \  {1}/{\lambda },\ \ \
 n_{\Lambda} (R) := \card \left\{ \lambda \in \Lambda \ \big| \ \left| \lambda \right| < R \right\},\ \ R > 0,
\end{equation*}
and  $ \card A \in \Bb{N}_{0} \cup \{+\infty \}$ denotes   the number of elements in a set $A$.
Then, all functions $f\in {\mathcal{E}}_{0} (\Bb{R})$ satisfying $\Lambda_{f} = \Lambda$
are given by the following formula:
\begin{equation*}\label{plm16}
 f(z) = A \lim_{R \to \infty} \prod\nolimits_{\lambda \in \Lambda \cap (-R, R)}\left(1 - {z}/{\lambda }\right),
 \ \ A \in \Bb{R}\setminus \{0\},\ \ z \in \Bb{C},
\end{equation*}
where $f (0) = A \neq 0$. Thus,
\begin{equation}\label{plm17}
  B(z) = B(0) \lim_{R \to \infty}  \prod\nolimits_{\lambda \in \Lambda_{B} \cap (- R , R)} \left(1 -{z}/\lambda \right),\ \ \ z \in \Bb{C}.
\end{equation}
and it follows from $\lim_{R \to +\infty } {n_{B} (R)}/{R} = 0$ that $\sum_{\lambda \in \Lambda_{B}}1/\lambda^{2} < \infty$.

\smallskip
Denote $\Lambda_{D} := \{d_{\lambda}\}_{\lambda\in \Lambda_{B}}$. Since
$\Lambda_{B} = \{\lambda\}_{\lambda\in \Lambda_{B}}$ satisfies the conditions of Lindel\"{o}f's theorem,
they are also met by the set $\Lambda_{D}$ because
\begin{equation*}
 \left| d_{\lambda} - \lambda \right| \leq \frac{\rho_{\delta} }{\delta^{2} \lambda^{2} }, \  \ \ \lambda\in \Lambda_{B},
\end{equation*}
by virtue of \eqref{9j} and the inequality
\begin{equation}\label{plm18a}
  \exp(-x) \leq 1/x^{2},\ \ \   x > 0.
\end{equation}
Therefore, $\Lambda_{D}$ is the set of all zeros of the entire function
\begin{equation}\label{plm19}
D (z) := \lim_{R \to \infty}  \prod\limits_{\substack{\lambda \in \Lambda_{B}\\
d_{\lambda} \in (- R, R)}} \left(1 - {z}/d_{\lambda}\right),\ \ z \in \Bb{C},
\end{equation}
which belongs to the class ${\mathcal{E}}_{0} (\Bb{R})$.

Let $m$ denote the Lebesgue measure on $\Bb{R}$. Then it follows from \eqref{plm5}, \eqref{plm15a} and \eqref{plm18a} that
\begin{equation*}
  m \left(\Lambda_{B}^{\!\Delta}\right) \leq 4 \sqrt{\rho_{\delta}} \sum\limits_{\lambda \in \Lambda_{B}}
  \mathrm{e}^{-\varepsilon |\lambda|} \leq 4\varepsilon^{- 2} {\sqrt{\rho_{\delta}}}{}\sum\limits_{\lambda \in \Lambda_{B}}1/\lambda^{2}<\infty.
\end{equation*}
Hence, the set
\begin{equation*}\label{plm19c}
 \Bb{R}_{B}^{+} := [0, +\infty) \setminus  \left( \Lambda_{B}^{\!\Delta} \cup - \Lambda_{B}^{\!\Delta}\right)
\end{equation*}
is unbounded and in view of \eqref{plm9}, \eqref{plm4} and \eqref{plm15a} we have
\begin{equation}\label{plm19a}
   \left\{ \lambda \, | \, \lambda \in \Lambda_{B} \cap (-R, R)\right\} = \left\{ \lambda \, | \, \lambda \in \Lambda_{B}, \,  \
   d_{\lambda} \in (- R ,  R )  \right\},\ \  \ R \in \Bb{R}_{B}^{+}.
\end{equation}

\vspace{0.25cm}
\subsection{}
Let us estimate $ |B^{\,\prime}(\lambda_{0})| / |D^{\,\prime}(d_{\lambda_{0}})|$ for arbitrary $\lambda_{0} \in \Lambda_{B}$.
It follows from \eqref{lem32}, \eqref{lem33}, \eqref{plm2} and \eqref{plm5} that
\begin{equation*}
\left| \frac{B (x)}{x  -  \lambda}\right| \geq \left| B^{\,\prime}(\lambda)\right| / 2,\ \ \ x \in \left[\lambda - 4 \Delta_{\lambda},
\lambda   +   4 \Delta_{\lambda} \right],\ \ \ \lambda\in \Lambda_{B},
\end{equation*}
and therefore, by \eqref{plm9}, we have
\begin{equation*}
  \left| B^{\,\prime} (\lambda_{0})\right| \leq \frac{2}{|\lambda_{0} |}
 \frac{\left|B (d_{\lambda_{0}})\right|}{ \left|1 - \dfrac{d_{\lambda_{0}}}{\lambda_{0} } \right|} \ .
\end{equation*}
Then, by \eqref{plm17}, \eqref{plm19} and \eqref{plm19a},
\begin{align}\nonumber
 \frac{|B^{\,\prime } (\lambda_{0})|}{|D^{\,\prime } (d_{\lambda_{0}})|} &  \leq \frac{2}{|\lambda_{0} |}
 \frac{\left|B (d_{\lambda_{0}})\right|}{ \left|1 - \dfrac{d_{\lambda_{0}}}{\lambda_{0} } \right| |D^{\,\prime }
 (d_{\lambda_{0}})|} \\ & = \frac{2 |d_{\lambda_{0}}|
|B (0)|}{|\lambda_{0} |} \lim_{\substack{R \to\, +\infty\\ R\in \Bb{R}_{B}^{+}}}
\prod\limits_{\substack{\lambda \in \Lambda_B \cap \,( - R , R ) \\
\lambda \neq \lambda_{0} } } \quad \left| \frac{1 - {d_{\lambda_{0}}}/{\lambda } }{ 1- d_{\lambda_{0}} /d_{\lambda}}\right|.
\label{plm22}\end{align}

The relations \eqref{plm15a} and \eqref{plm7c} imply that $0 \notin
 [\lambda - 2 \Delta_{\lambda}, \, \lambda\! + 2 \Delta_{\lambda} \, ]$
and therefore $|\lambda| \leq 2 \Delta_{\lambda}$, which together with the consequence
$|d_{\lambda}| \leq |\lambda| + \Delta_{\lambda}^{2}$ of \eqref{plm9} yields in view of \eqref{plm5}
$|d_{\lambda}/\lambda | \leq 1 + \Delta_{\lambda}^{2}/|\lambda| \leq 1+ \Delta_{\lambda}/2 \leq 2$ for every
$\lambda \in \Lambda_{B}$. Thus,  in \eqref{plm22} we have $|d_{\lambda_{0}}|/|\lambda_{0}| \leq 2$.

\smallskip
Setting in Lemma~\ref{lem1}, $x=d_{\lambda_{0}}$, $a=\lambda$, $b =d_{\lambda}$  and
$\Delta = \Delta_{\lambda}$ with $\lambda_{0}$ and $\lambda$ taken from \eqref{plm22},
we obtain the validity of the conditions \eqref{lem11},
\begin{equation*} \label{plm23}
d_{\lambda} \in (\lambda-\Delta_{\lambda}^{2}, \lambda+\Delta_{\lambda}^{2}),\ \  \
0\notin (\lambda-\Delta_{\lambda} , \lambda+\Delta_{\lambda} ), \ \ \ d_{\lambda_{0}}
\notin (\lambda-2\Delta_{\lambda} , \lambda+2\Delta_{\lambda} ),
\end{equation*}
as a consequence of \eqref{plm9}, \eqref{plm15a}, \eqref{plm7c},
and \eqref{plm10}. Hence, the factors in \eqref{plm22} satisfy
\begin{equation*}
\left|\left(1 - d_{\lambda_{0}}/{\lambda } \right)
\left( 1- d_{\lambda_{0}} /d_{\lambda}\right)^{-1}\right| \leq
\left(1 + \Delta_{\lambda}\right)^{2} =
\left(1 + \sqrt{\rho_{\delta}} \ \mathrm{e}^{{\fo{- \varepsilon |\lambda|}}}\right)^{2},
\end{equation*}
by virtue of \eqref{plm5}. It follows therefore from  \eqref{plm22} that
\begin{equation}\label{plm26}
  \frac{|B^{\,\prime } (\lambda_{0})|}{|D^{\,\prime}(d_{\lambda_{0}})|} \leq C_{\delta} := 4 |B (0)| \prod\limits_{\lambda \in \Lambda_B}
  \left(1 + \sqrt{\rho_{\delta}} \ \mathrm{e}^{{\fo{- \varepsilon |\lambda|}}}\right)^{2} < \infty,\ \ \  \lambda_{0} \in \Lambda_B,
\end{equation}
where the product above is finite in view of \eqref{plm18a}, \eqref{plm3a}(a) and
$\sum_{\lambda \in \Lambda_{B}}1/\lambda^{2} < \infty$. Lemma~\ref{lm1} is proved and the formulas \eqref{plm26}, \eqref{plm2} together with the reasoning of Subsection~\ref{explicit} establish the explicit expressions for the constants $\rho_{\delta}$ and $C_{\delta}$
in \eqref{9j} and in \eqref{11j}.

\vspace{0.5cm}
\section{\texorpdfstring{Proof of Corollary~\ref{cor1}}{Proof of Corollary 1} }\label{corollary1}
\vspace{0.25cm}
We first prove \eqref{fz1cor1}.  It follows from
\begin{align*}
   \int\limits_{-1}^{1} {\rm{e}}^{- \tfrac{1}{1-t^{2}}}  \mathrm{d} t &  =\frac{1}{\rm{e}}   \int\limits_{0}^{\infty} \frac{{\rm{e}}^{- t} \mathrm{d} t}{\sqrt{t}(t+1)^{3/2} } =
    \frac{2}{{\rm{e}}}   \int\limits_{0}^{\infty} {\rm{e}}^{- t} \mathrm{d}  \sqrt{\frac{t}{t+1}}
  = \frac{2}{{\rm{e}}} \int\limits_{0}^{\infty} {\rm{e}}^{- t}   \sqrt{\frac{t}{t+1}} \mathrm{d} t   \\
  &= - \frac{2}{{\rm{e}}} \frac{\mathrm{d}}{\mathrm{d} x}  \int_{0}^{\infty} \frac{{\rm{e}}^{- x t}  \mathrm{d} t}{\sqrt{t (t+1)}} \Big|_{x=1} =
      - \frac{2}{{\rm{e}}} \frac{\mathrm{d}}{\mathrm{d} x} \mathrm{e}^{x/2} \int_{1}^{\infty} \frac{{\rm{e}}^{- (x/2) t}  \mathrm{d} t}{\sqrt{t^{2}-1)}} \Big|_{x=1}
\end{align*}

\noindent
and \cite[(19), p.82]{erd} that
\begin{gather*}
 \kappa = - \frac{2}{{\rm{e}}} \frac{\mathrm{d}}{\mathrm{d} x}\, {\rm{e}}^{x/2} K_{0} (x/2) \Big|_{x=1} =
    - \frac{K_{0} (1/2)  + K_{0}^{\, \prime} (1/2)}{\sqrt{{\rm{e}}}} = \frac{K_{1} (1/2) - K_{0} (1/2)}{\sqrt{{\rm{e}}}} \ ,
    \end{gather*}

\noindent
by virtue of \cite[(21), p.79]{erd}. The values in \cite[p.417]{abr}, $\mathrm{e}^{0.5}K_{0} (0.5) =1.52410... $ and
$\mathrm{e}^{0.5}K_{1} (0.5) = 2.73100...$  finish the proof of \eqref{fz1cor1}. Similarly, we obtain
\begin{gather}\label{pf5cor1}
\int\limits_{-1}^{ 1}   \frac{2 |t| }{\left(t^{2} -1\right)^{2}} \ \mathrm{exp}\left(- \frac{1}{1 -t^{2}}  \right)  {\rm{d}} t = \frac{2}{{\rm{e}}} \  ,
\ \
\int\limits_{-1}^{1}    \frac{
 t^{2}}{\left(t^{2} -1\right)^{2}}\ \mathrm{exp}\left(- \frac{1}{1 -t^{2}}  \right)   {\rm{d}} t= \frac{\kappa }{2} \ .
\end{gather}

Let the function  $ \phi_{\varepsilon}$ be defined in \eqref{fz0cor1}. In order to prove Corollary~\ref{cor1}, we observe that all  constant functions belong to the set  $C^{\infty}(\Bb{R})$  and therefore we can apply Lemma~\ref{lem4} to the function
$4 {\rm{e}}^{\varepsilon} \phi_{\varepsilon}$ which coincides with the function $f_{\omega}$ in \eqref{f1lem4} for $\omega \equiv 1$ and $f (x)=
 \exp (- \varepsilon |x|)$.
Thus,   \eqref{f2lem4} and   \eqref{fz0cor1} yield for every $x\in \Bb{R}$ that
\begin{gather}\label{pf1cor1}
   \phi_{\varepsilon}^{\, \prime} (x) =  \frac{{\rm{e}}^{-\varepsilon} }{   4 \kappa }
\int\limits_{-1}^{ 1} {\rm{e}}^{ - \varepsilon |x+t|}   \frac{2 t {\rm{e}}^{- \tfrac{1}{1 -t^{2}} }}{\left(1 - t^{2} \right)^{2}} \,    {\rm{d}} t, \ \
  \phi_{\varepsilon} (x) =\frac{{\rm{e}}^{-\varepsilon} }{   4 \kappa } \int\limits_{-1}^{1}\mathrm{e}^{- \varepsilon | x +t | } {\rm{e}}^{- \tfrac{1}{1-t^{2}}}
   {\rm{d}} t .
  \end{gather}

\noindent
It follows from \eqref{pf1cor1}, \eqref{pf5cor1}, \eqref{fz1cor1} and
\begin{gather*}
   - |x| - 1   \leq    -  |x+t| \leq  1- | x  | , \ \ |t| \leq 1  , \ t, x \in \Bb{R},
\end{gather*}

\noindent
that
\begin{gather}\label{pf3cor1}
 \left|\phi_{\varepsilon}^{\, \prime} (x)\right| \leq (5/12)\, {\rm{e}}^{ - \varepsilon |x|} , \ \ \ \
 ( {\rm{e}}^{-2 \varepsilon}/4) \, {\rm{e}}^{ - \varepsilon |x|} \leq  \phi_{\varepsilon} (x) \leq (1/4)\,  {\rm{e}}^{ - \varepsilon |x|} , \ x \in \Bb{R} .
\end{gather}

Let $W_{\varepsilon}$ be defined as in \eqref{mp1y} with $\omega = \phi_{\varepsilon}$. Then by Theorem~\ref{th1},  $W_{\varepsilon} \in C^{\infty} (\Bb{R}) \cap \mathcal{W}^{{\rm{reg}}} (\Bb{R})$ and $W_{\varepsilon} (x) \geq w (x) + \mathrm{e}^{- \varepsilon |x|}$ for all $x\in \Bb{R}$. Furthermore,
$W_{\varepsilon}$ is equal to the function $f_{\omega}$ in \eqref{f1lem4} for $\omega = \phi_{\varepsilon}$ and
$f = \Omega_{1/2}$. Thus, by  \eqref{f2lem4}, we obtain
\begin{align} \nonumber
  W_{\varepsilon}^{\, \prime} (x)&  =   \frac{1}{   \kappa \, \phi_{\varepsilon} (x)}
\int\limits_{-1}^{ 1}  \Omega_{1/2} (x+t\phi_{\varepsilon} (x))  \frac{2 t }{\left(t^{2} -1\right)^{2}} \ \mathrm{exp}\left(- \frac{1}{1 -t^{2}}  \right)  {\rm{d}} t -   \frac{\phi_{\varepsilon}^{\,\prime} (x)}{\phi_{\varepsilon} (x)} W_{\varepsilon} (x) \\ &
  + \frac{2 \, \phi_{\varepsilon}^{\,\prime} (x)}{ \kappa \,  \phi_{\varepsilon} (x)}  \int\limits_{-1}^{1} \Omega_{1/2} (x+t\phi_{\varepsilon} (x))   \frac{
 t^{2}}{\left(t^{2} -1\right)^{2}}\ \mathrm{exp}\left(- \frac{1}{1 -t^{2}}  \right)   {\rm{d}} t \  .
\label{pf4cor1}\end{align}

According to \eqref{mp4x} and \eqref{mp4}, $ \Omega_{1/2} (x+t\phi_{\varepsilon} (x)) \leq w_{\varepsilon} (x)$ and $W_{\varepsilon} (x)  \leq w_{\varepsilon} (x)$ for all $|t|\leq 1$ and $x \in \Bb{R}$. Therefore, it follows from \eqref{pf5cor1}, \eqref{pf3cor1} and \eqref{pf4cor1} that
\begin{align*}
  \left| W_{\varepsilon}^{\, \prime} (x)\right| & \leq w_{\varepsilon} (x) \left[\frac{ 4 \mathrm{e}^{2 \varepsilon}  \mathrm{e}^{\varepsilon |x|}}{ \kappa }\cdot \frac{2}{{\rm{e}}} +
 \frac{10  \mathrm{e}^{2 \varepsilon}}{6}   + \frac{2}{\kappa } \cdot \frac{10 \mathrm{e}^{2 \varepsilon} }{6}\cdot  \frac{\kappa }{2} \right]
\\ & \leq 10  \mathrm{e}^{2 \varepsilon} \mathrm{e}^{\varepsilon |x|} w_{\varepsilon}(x)\leq 74 \mathrm{e}^{\varepsilon |x|} w_{\varepsilon}(x) ,
\end{align*}

\noindent
which completes the proof of Corollary~\ref{cor1}.

\section*{Acknowledgment}

The authors thank the referee for important remarks resulting in the
inclusion of  Corollaries~\ref{cor1} and~\ref{cor2} in the paper.

\end{document}